\UseRawInputEncoding
\documentclass{article}
\usepackage{amssymb,amscd,epsf,graphics,epsfig,latexsym}
\usepackage{stmaryrd,graphicx}
\usepackage{mathrsfs,amsmath,amsthm}
\usepackage{float}
\usepackage[all,cmtip]{xy}
\DeclareMathAlphabet{\mathpzc}{OT1}{pzc}{m}{it}
\usepackage{cleveref}
\usepackage{color}

\DeclareMathAlphabet{\mathpzc}{OT1}{pzc}{m}{it}

\newcommand{\loopnx}{\Omega^{n}(X,x_0)}
\newcommand{\met}{\text{met}}
\newcommand{\piomet}{\pi_{1}^{\met}}
\newcommand{\pinmet}{\pi_{n}^{\met}}

\newcommand{\pishapex}{\pi_{1}^{\text{sh}}(X,x_0)}

\newcommand{\pinshape}{\pi_{n}^{\text{sh}}}
\newcommand{\pinshapex}{\pi_{n}^{\text{sh}}(X,x_0)}
\newcommand{\pinqtop}{\pi_{n}^{\text{qtop}}}
\newcommand{\pinqtopx}{\pi_{n}^{\text{qtop}}(X,x_0)}

\renewcommand{\span}{\pi_{n}^{\text{Sp}}}

\newcommand{\pinx}{\pi_{n}(X,x_0)}

\newcommand{\shapen}{\check{\pi}_{n}(X,x_0)}

\newcommand{\scru}{\mathscr{U}}

\newcommand{\scrv}{\mathscr{V}}

\newcommand{\bbe}{\mathbb{E}}

\newcommand{\bbn}{\mathbb{N}}

\newcommand{\bbr}{\mathbb{R}}

\newcommand{\bbz}{\mathbb{Z}}

\newcommand{\nerveu}{N(\mathscr{U})}
\newcommand{\nervev}{N(\mathscr{V})}

\newcommand{\ui}{[0,1]}
\newcommand{\cov}{\text{cov}}

\newtheorem{theorem}{Theorem}[section]
\theoremstyle{definition}\newtheorem{definition}[theorem]{Definition}
\newtheorem{lemma}[theorem]{Lemma}
\newtheorem{proposition}[theorem]{Proposition}
\newtheorem{corollary}[theorem]{Corollary}

\newtheorem{example}[theorem]{Example}

\newtheorem{problem}[theorem]{Problem}
\newtheorem{remark}[theorem]{Remark}

\begin{document}

\title{A natural pseudometric on homotopy groups of metric spaces}
\author{Jeremy Brazas and Paul Fabel}

\maketitle

\begin{abstract}
For a path-connected metric space $(X,d)$, the $n$-th homotopy group $\pi_n(X)$ inherits a natural pseudometric from the $n$-th iterated loop space with the uniform metric. This pseudometric gives $\pi_n(X)$ the structure of a topological group and when $X$ is compact, the induced pseudometric topology is independent of the metric $d$. In this paper, we study the properties of this pseudometric and how it relates to previously studied structures on $\pi_n(X)$. Our main result is that the pseudometric topology agrees with the shape topology on $\pi_n(X)$ if $X$ is compact and $LC^{n-1}$ or if $X$ is an inverse limit of finite polyhedra with retraction bonding maps.
\end{abstract}

\section{Introduction}

There are many ways to enrich the $n$-th homotopy group $\pi_n(X,x_0)$ of a based topological space $(X,x_0)$ with geometric or topological structure that remembers local features of the space $X$, which are ``unseen" by the usual group-theoretic structure. For example, the natural quotient topology \cite{BFqtop}, the $\tau$-topology \cite{Braztopgrp}, the Spanier topology \cite{AJMPT,AcetiBrazas}, the whisker topology \cite{AJMPTwhisker}, and variations on these (e.g. coreflections in a convenient category). Of particular interest is the initial topology on $\pi_n(X,x_0)$ with respect to the canonical homomorphism $\Psi_{n}:\pinx\to \shapen$ to the $n$-th shape homotopy group. This topology is often referred to as the \textit{shape topology} and we denote the resulting topological group as $\pinshape(X,x_0)$. The shape topology gives $\pinshape(X,x_0)$ the structure of a pro-discrete group and is closely related to Dugundji's pre-shape-theory approach in \cite{Dug50}. While these topologies, described in more detail in Section \ref{sectiontopologies}, are well-suited for characterizing various properties in the topological setting, they often forget geometric features determined by a choice of metric.

In this paper, we show that a metric space $(X,d)$ with basepoint $x_0\in X$ determines a pseudometric $\rho$ on $\pi_n(X,x_0)$. While the vast majority of papers on topologized homotopy groups focus on the fundamental group, our results hold in arbitrary dimensions. The resulting pseudometric group $(\pi_n(X,x_0),\rho)$ gives the $n$-th homotopy group the structure of a topological group (Proposition \ref{normalsubgroups}), which we denote as $\pinmet(X,x_0)$. For general metric spaces, $\rho$ depends entirely on the given metric $d$ and the topology on $\pinmet(X,x_0)$ induced by $\rho$ may also vary (see Example \ref{puncturedplane}). However, we show that when $X$ is compact, the topology of $\pinmet(X,x_0)$ induced by $\rho$ is independent of the choice of metric $d$  (Theorem \ref{independentthm}). We prove these results in Section \ref{sectionbasictheory} after establishing the basic theory of $\pinmet(X,x_0)$

In Section \ref{sectioncomparewithshape}, we prove the following theorem, which compares the topology of  $\pinmet(X,x_0)$ with the shape topology. Recall that a space $X$ is $LC^{n}$ if for every open neighborhood $U$ of a point $x\in X$, there is an open neighborhood $V$ of $x$ such that $V\subseteq U$ and such that every map $f:S^{k}\to V$ where $0\leq k\leq n$ is null-homotopic in $U$. For all $n\geq 0$, there are compact metric spaces which are $LC^n$ but not locally $n$-connected. 

\begin{theorem}\label{mainthm}
Let $(X,d)$ be a path-connected compact metric space and $n\geq 1$. The topology induced by the natural pseudometric $\rho$ on $\pi_n(X,x_0)$ is at least as fine as the shape topology. Moreover, these two topologies agree in the following two cases:
\begin{enumerate}
\item if $X$ is $LC^{n-1}$,
\item if $X=\varprojlim_{j\in\bbn}(X_j,r_{j+1,j})$ is an inverse limit of finite polyhedra where the bonding maps $r_{j+1,j}:X_{j+1}\to X_j$ are retractions.
\end{enumerate}
\end{theorem}

To prove this result, we show that for a compact metric space $(X,d)$ the pseudometric topology on $\pi_n(X,x_0)$ lies between the shape topology and a third topology called the \textit{Spanier topology} (Definition \ref{defSpaniertopology}). For Case (1), we apply a result from \cite{AcetiBrazas} to conclude that, in the $LC^{n-1}$ case, the Spanier and shape topologies (and thus all three topologies) agree. We prove Case (2) separately using the fact that we may alter the metric on $X$ without affecting the topologies of $\pinmet(X,x_0)$ or $\pinshape(X,x_0)$.

While Case (2) may appear structurally restrictive, it provides insight to some prominent examples. For example, if $X$ is an infinite shrinking wedge of finite polyhedra $X_j$, then $X$ my be identified with $\varprojlim_{k\in\bbn}\bigvee_{j=1}^{k}X_j$, which has retraction bonding maps (see Example \ref{shrinkingwedge}). While the higher homotopy groups of spaces of the form $X$ are only known in some special cases \cite{EK00higher}, Theorem \ref{maincor} implies that $\pinmet(X,x_0)\cong \pinshape(X,x_0)$ for all $n\geq 1$.

The upshot of Case (1) of Theorem \ref{mainthm} is that in the presence of local $(n-1)$-connectedness conditions, one can often characterize the pseudometric topology by appealing to shape-theoretic methods, which are likely to be simpler than an analysis of the uniform metric on the $n$-th loop space. We can modify an example from \cite{CMRZZ08} to show Theorem \ref{mainthm} does not hold for $n\geq 1$ without the $LC^{0}$ condition (see Example \ref{cylinderexample}). The authors do not know of a higher dimensional counterexample, that is, a Peano continuum $X$ and $n\geq 2$ for which the pseudometric topology on $\pi_n(X,x_0)$ is strictly finer than the shape topology. According to Theorem \ref{mainthm} such an example must fail to be $LC^{n-1}$ and also cannot be an inverse limit of retracts of finite polyhedra. This leaves the following problem.

\begin{problem}\label{problem}
For $n\geq 2$, give an example of a based Peano continuum $(X,x_0)$ for which the topology of $\pi_{n}^{\text{met}}(X,x_0)$ is strictly finer than that of $\pi_{n}^{\text{sh}}(X,x_0)$.
\end{problem}

\section{Preliminaries}\label{sectionprelim}

Throughout this paper, $(X,d)$ will denote a path-connected metric space with basepoint $x_0\in X$. The unit interval is denoted $\ui$, the unit $n$-disk is denoted $D^n=\{\mathbf{x}\in\bbr^n\mid \|\mathbf{x}\|\leq 1\}$ and $S^{n-1}=\partial D^n$ is the unit $(n-1)$-sphere. The latter two spaces have basepoint $d_0=(1,0,\dots ,0)$. The $n$-th homotopy group of $(X,x_0)$ is denoted $\pi_n(X,x_0)$. If $f:(X,x_0)\to (Y,y_0)$ is a based map, then $f_{\#}:\pi_n(X,x_0)\to \pi_n(Y,y_0)$ is the induced homomorphism.

Let $\Omega^n(X,x_0)$ be the space of maps $\alpha:(\ui^n,\partial \ui^n)\to (X,x_0)$ based at $x_0$ (which we call \textit{$n$-loops}) with the metric of uniform convergence. It is well-known that the uniform metric topology agrees with the usual compact-open topology on $\loopnx$ (when convenient, we may identify $\Omega^n(X,x_0)$ with the space of based maps $(S^n,d_0)\to (X,x_0)$ also with the uniform metric). Let $\pi:\loopnx\to\pi_n(X,x_0)$, $\pi(\alpha)=[\alpha]$ denote the canonical surjection taking a map $\alpha$ to its homotopy class. 

Given $\alpha\in\Omega^n(X,x_0)$, let $\alpha^{-}(t_1,t_2,\dots,t_n)=\alpha(1-t_1,t_2,\dots,t_n)$ denote the reverse of $\alpha$ and if $\alpha_1,\alpha_2,\dots,\alpha_n$ is a sequence of $n$-loops, then $\alpha_1\cdot \alpha_2\cdots \alpha_n$ is the usual $n$-fold concatenation defined as $\alpha_i$ on $\left[\frac{i-1}{n},\frac{i}{n}\right]\times \ui^{n-1}$. Generally, $c_{x_0}\in \Omega^n(X,x_0)$ will denote the constant map at $x_0$ so that $e=[c_{x_0}]$ serves as the identity element of $\pi_n(X,x_0)$. 


\section{Topologies on homotopy groups}\label{sectiontopologies}

Here, we briefly recall some previously studied topologies on the homotopy groups.

\subsection{The quotient topology}

Let $\pinqtopx$ denote the $n$-th homotopy group with the quotient topology with respect to the canonical map $\pi:\loopnx\to\pi_n(X,x_0)$, that is, the finest topology on $\pi_n(X,x_0)$ such that $\pi$ is continuous. In particular $A\subseteq \pinqtopx$ is open (closed) if and only if $\pi^{-1}(A)$ is open (closed) in $\loopnx$. It is known that this topology gives $\pi_n(X,x_0)$ the structure of a \textit{quasitopological group} \cite{AT08} (in the sense that inversion is continuous and multiplication is continuous in each variable) which can fail to be a topological group \cite{Brazfretopgrp,Fabel11HE,Fabcgqtop} even when $X$ is a compact metric space. For a general study of $\pi_1$ with the quotient topology, we refer the reader to \cite{BFqtop}. It is known that if $X$ is locally $(n-1)$-connected and semilocally $n$-connected, then $\Omega^n(X,x_0)$ is locally path-connected \cite{Wada}. In such a situation, $\pinqtop(X,x_0)$ is guaranteed to be discrete (see also \cite{Calcut} in dimension $n=1$). In particular, if $X$ is a polyhedron, manifold, or CW-complex, then $\pinqtop(X,x_0)$ is a discrete group for all $n\geq 1$.

\subsection{The ``tau topology"}

In \cite{Braztopgrp}, it was observed that for any quasitopological group $G$, there is a finest group topology on the group $G$, which is coarser than that of $G$. The resulting topological group is denoted $\tau(G)$. In other words, the category of topological groups is a reflective subcategory of the category of quasitopological groups, where $\tau$ is the reflection functor. In the case of homotopy groups, the $\tau$-reflection $\pi_{n}^{\tau}(X,x_0)=\tau(\pinqtopx)$ is a topological group. The topology of $\pi_{n}^{\tau}(X,x_0)$ is coarser than that of $\pinqtopx$ and agrees with that of $\pinqtopx$ if and only if $\pinqtopx$ is a topological group. The $\tau$-topology is the finest \textit{group} topology on $\pinx$ for which $\pi:\loopnx\to\pinx$ is continuous. In \cite{VZtau}, an analogous construction is given for the ``universal path space," i.e. the set of path-homotopy classes of paths in $X$ starting at $x_0$.

\subsection{The shape topology}

We give a few details regarding the construction of the $n$-th shape homotopy group and refer the reader to \cite{MS82} for a detailed treatment of shape theory. Let $\cov(X)$ be the directed set of pairs $(\scru,U_0)$ where $\scru$ is a locally finite open cover of $X$ and $U_0$ is a distinguished element of $\scru$ containing $x_0$. Here, $\cov(X)$ is directed by refinement. Given $(\scru,U_0)\in \cov(X)$ let $N(\scru)$ be the abstract simplicial complex which is the nerve of $\scru$. In particular, $\scru$ is the vertex set of $U$ and the n vertices $U_1,\dots,U_n$ span an n-simplex $\Leftrightarrow$ $\bigcap_{i=1}^{n}U_i\neq \emptyset$. The geometric realization $|\nerveu|$ is a polyhedron and thus $\pi_n(|\nerveu|,U_0)$ may be regarded naturally as a discrete group, e.g. if it is given the quotient topology.

Given a pair $(\scrv,V_0)$ which refines $(\scru,U_0)$, a simplicial map $p_{\scru\scrv}:|\nervev|\to |\nerveu|$ is constructed by sending a vertex $V\in \scrv$ to some $U\in\scru$ for which $V\subseteq U$ (in particular, $V_0$ is mapped to $U_0$) and extending linearly. The map $p_{\scru\scrv}$ is unique up to homotopy and thus induces a unique homomorphism $p_{\scru\scrv\#}:\pi_n(|\nervev|,V_0)\to \pi_n(|\nerveu|,U_0)$. The inverse system \[(\pi_n(|\nerveu|,U_0),p_{\scru\scrv\#},\cov(X))\] of discrete groups is the \textit{$n$th pro-homotopy group} and the limit $\check{\pi}_n(X,x_0)$ (topologized with the usual inverse limit topology) is the \textit{n-th shape homotopy group}.

Given a partition of unity $\{\phi_{U}\}_{U\in \mathscr{U}}$ subordinated to $\scru$ and such that $\phi_{U_0}(x_0)=1$, a map $p_{\scru}:X\to |\nerveu|$ is constructed by taking $\phi_{U}(x)$ (for $x\in U$, $U\in \scru$) to be the barycentric coordinate of $p_{\mathscr{U}}(x)$ corresponding to the vertex $U$. The induced continuous homomorphism $p_{\scru\#}:\pi_n(X,x_0)\to \pi_{n}(|\nerveu|,U_0)$ satisfies $p_{\scru\#}\circ p_{\scru\scrv\#}=p_{\scrv\#}$ whenever $(\scrv,V_0)$ refines $(\scru,U_0)$. Thus there is a canonical, continuous homomorphism $\Psi_n:\pi_n(X,x_0)\to \shapen$ to the $n$-th shape homotopy group, given by $\Psi_n([\alpha])=([p_{\scru}\circ \alpha])_{\scru}$.

\begin{definition}
The \textit{shape topology} on $\pi_n(X,x_0)$ is the initial topology with respect to the $n$-th shape homomorphism $\Psi_n:\pi_n(X,x_0)\to \shapen$ described above. Let $\pinshape(X,x_0)$ denote $\pi_n(X,x_0)$ equipped with the shape topology.
\end{definition}

The shape topology is characterized as follows: $A\subset \pinshape(X,x_0)$ is open (resp. closed) $\Leftrightarrow$ $A=\Psi_{n}^{-1}(B)$ for an open (resp. closed) set $B\subset \shapen$. Equivalently, a neighborhood base at the identity element is formed by the normal subgroups $\ker (p_{\scru\#}:\pi_n(X,x_0)\to \pi_n(|N(\scru)|,U_0))$, $(\scru,U_0)\in \cov(X)$. Since $\shapen$ is a topological group, it follows immediately that $\pinshape(X,x_0)$ is a topological group. 
\begin{proposition} \cite[3.24]{Braztopgrp}
For any space $X$, the shape topology of $\pinshapex$ is coarser than that of $\pi_{n}^{\tau}(X,x_0)$.
\end{proposition}
%
%
\begin{definition}
We say a space $X$ is $\pi_n$\textit{-shape injective} if $\Psi_n:\pinx\to \shapen$ is a monomorphism.
\end{definition}
By construction, $\pinshapex$ is Hausdorff $\Leftrightarrow$ $X$ is $\pi_n$-shape injective.
\section{A pseudometric on homotopy groups}\label{sectionbasictheory}

Let $(X,d)$ be a path-connected metric space and consider the uniform metric \[\mu(\alpha,\beta)=\sup_{t\in \ui^n}\{d(\alpha(t),\beta(t))\}\] on $\Omega^n(X,x_0)$. Observe that \[\mu(\alpha\cdot \alpha ',\beta\cdot \beta ')=\max\{\mu(\alpha,\beta),\mu(\alpha ',\beta ')\}\] and $\mu(\alpha,\beta)=\mu(\alpha^{-},\beta^{-})$. We consider the following function $\rho:\pinx\times\pinx\to [0,\infty)$ on the homotopy group $\pinx$: \[\rho(a,b)=\inf\{\mu(\alpha,\beta)\mid\alpha\in a,\beta\in b\}.\] We will show below that $\rho$ is a pseudometric on $\pinx$. Certainly $\rho$ is symmetric and $\rho(a,a)=0$; however, a little more work is required to verify the triangle inequality. 

\begin{remark}
In general, if $(X,d)$ is a metric space, $\sim$ is an equivalence relation, and $Y=X/\mathord{\sim}$, then the definition $\rho:Y\times Y\to [0,\infty)$, $\rho(a,b)=\inf\{d(\alpha,\beta)\mid \alpha\in a,\beta\in b\}$ need not satisfy the triangle inequality. In our situation, we must make use of the group structure of $\pinx$ and the nature of the uniform metric $\mu$ in order to verify the triangle inequality.
\end{remark}

In the next three results, we assume $(X,d)$ is an arbitrary metric space, $x_0\in X$, and $\mu$ and $\rho$ are defined as above.

\begin{lemma}[Isometric Inversion]
For all $a,b\in \pinx$, we have $\rho(a,b)=\rho(a^{-1},b^{-1})$.
\end{lemma}

\begin{proof}
Since $\mu(\alpha,\beta)=\mu(\alpha^{-},\beta^{-})$ for all $\alpha\in a$, $\beta\in b$, it is clear that \[\rho(a,b)=\rho([\alpha^{-}],[\beta^{-}])=\rho(a^{-1},b^{-1}).\]
\end{proof}

\begin{lemma}[Isometric translations]\label{translations}
For all $a,b,c\in \pinx$, we have $\rho(a,b)=\rho(ac,bc)=\rho(ca,cb)$.
\end{lemma}
\begin{proof}
Fix $a,b,c\in \pinx$ and $\alpha\in a$, $\beta\in b$, and $\gamma\in c$. We have $\mu(\alpha\cdot \gamma,\beta\cdot \gamma)=\mu(\alpha,\beta)$. It follows that \[\rho(a,b)=\rho([\alpha\cdot \gamma],[\beta\cdot \gamma])=\rho(ac,bc).\] The symmetric argument gives $\rho(a,b)=\rho(ca,cb)$.
\end{proof}
\begin{lemma}\label{multdist}
For all $a,b\in \pinx$, we have $\rho(ab,e)\leq \max\{\rho(a,e),\rho(b,e)\}$
\end{lemma}
\begin{proof}
Suppose $a,b\in \pinx$ and $\epsilon>0$. Find $\alpha\in a$, $\beta\in b$, and $\gamma_1,\gamma_2\in e=[c_{x_0}]$ such that $\mu(\alpha,\gamma_1)<\rho(a,e)+\frac{\epsilon}{2}$ and $\mu(\beta,\gamma_2)<\rho(b,e)+\frac{\epsilon}{2}$. Note that
\begin{eqnarray*}
\mu(\alpha\cdot \beta,\gamma_1\cdot \gamma_2) &=& \max\{\mu(\alpha,\gamma_1),\mu(\beta,\gamma_2)\}\\
&<& \max\left\{\rho(a,e)+\frac{\epsilon}{2},\rho(b,e)+\frac{\epsilon}{2}\right\}\\
&< &\max\left\{\rho(a,e),\rho(b,e)\right\}+\epsilon
\end{eqnarray*}
Thus $\rho(ab,e)\leq \max\left\{\rho(a,e),\rho(b,e)\right\}$.
\end{proof}
\begin{theorem}
For any metric space $(X,d)$, the function $\rho:\pinx\times\pinx\to [0,\infty)$, defined above, is a pseudometric on $\pinx$.
\end{theorem}
\begin{proof} As noted above, it suffices to verify the triangle inequality. Let $a,b,c\in \pinx$. Using the previous three lemmas, we have:
\begin{eqnarray*}
\rho(a,c) &=& \rho(ac^{-1},e) \\
&=& \rho(ab^{-1}bc,e) \\
&\leq &\max\{\rho(ab^{-1},e),\rho(bc^{-1},e)\} \\
&= &\max\{\rho(ab^{-1},e),\rho(cb^{-1},e)\} \\
&\leq & \rho(ab^{-1},e)+\rho(cb^{-1},e)\\
&=& \rho(a,b)+\rho(c,b) \\
 &=& \rho(a,b)+\rho(b,c)
\end{eqnarray*}
\end{proof}
\begin{proposition}\label{normalsubgroups}
Let $(X,d)$ be any metric space and $x_0\in X$. Equipped with the topology induced by the pseudometric $\rho$, $\pinx$ is a topological group whose open balls $B_{\rho}(e,r)=\{a\in \pi_n(X,x_0)\mid \rho(e,a)<r\}$, $r>0$ are open normal subgroups. 
\end{proposition}
\begin{proof}
Since the open balls $B_{\rho}(e,r)$ form a neighborhood base at $e$ and translations are homeomorphisms (Lemma \ref{translations}), it will follow that $\pinx$ is a topological group once we show that $B_{\rho}(e,r)$ is an open normal subgroup.

Since $\rho(a,e)=\rho(a^{-1},e)$, $B_{\rho}(e,r)$ is closed under inversion. Additionally, by Lemma \ref{translations}, we have \[\rho(bab^{-1},e)=\rho(a,b^{-1}b)=\rho(a,e)\]for all $a,b\in \pinx$. Thus $B_{\rho}(e,r)$ is closed under conjugation (particularly when $n=1$). Finally, if $\rho(a,e))<r$ and $\rho(b,e)<r$, then $\rho(ab,e)\leq \max\left\{\rho(a,e),\rho(b,e)\right\}<r$ by Lemma \ref{multdist} and it follows that $B_{\rho}(e,r)$ is closed under multiplication.
\end{proof}

\begin{definition}
For a metric space $(X,d)$ and $x_0\in X$, let $\pinmet(X,x_0)$ denote the $n$-th homotopy group equipped with the topology induced by the pseudometric $\rho$. We call this topology the \textit{pseudometric topology (induced by $d$)}.
\end{definition}

Since $\pinmet(X,x_0)$ is a topological group, open subgroups of $\pinmet(X,x_0)$ are also closed. Therefore, $\pinmet(X,x_0)$ is zero-dimensional. On the other hand, $\pinmet(X,x_0)$ need not be Hausdorff, since the closed normal subgroup $\bigcap_{r>0}B_{\rho}(e,r)=\{a\in \pi_n(X,x_0)\mid \rho(a,e)=0\},$ is equal to $\overline{\{e\}}$, the closure of the identity element, which may be non-trivial. In particular, $\overline{\{e\}}$ is non-trivial if and only if there exist sequences $\{\alpha_k\}_{k\in\bbn}$ and $\{\beta_k\}_{k\in\bbn}$ in $\loopnx$ such that $[\alpha_k]=[\alpha_{k+1}]$, $[\beta_k]=[\beta_{k+1}]$ and $\displaystyle\lim_{k\to \infty}\mu(\alpha_k,\beta_k)=0$.
\begin{proposition}\label{uniformlyinduced}
Let $(X,d)$ and $(Y,d')$ be metric spaces respectively inducing psuedometric $\rho$ on $\pi_n(X,x_0)$ and $\rho '$ on $\pi_n(Y,y_0)$. If $f:(X,d)\to (Y,d')$ is a uniformly continuous map such that $f(x_0)=y_0$, then the induced homomorphism $f_{\#}:(\pinmet(X,x_0),\rho)\to (\pinmet(Y,y_0),\rho ')$ is uniformly continuous.
\end{proposition}
\begin{proof}
Let $\mu$ and $\mu '$ be the uniform metrics on $\Omega^n(X,x_0)$ and $\Omega^n(Y,y_0)$ respectively. Suppose $\epsilon>0$. There is a $\delta>0$ such that $\mu(\alpha,\beta)<\delta$ $\Rightarrow$ $\mu'(f\circ \alpha,f\circ \beta)<\epsilon/2$. Suppose $\rho(g,h)<\delta$ for $g,h\in \pinmet(X,x_0)$. There are $\alpha\in g$, $\beta\in h$ such that $\mu(\alpha,\beta)<\delta$. Thus $\mu'(f\circ \alpha,f\circ \beta)<\epsilon/2$. It follows that $\rho'(f_{\#}(g),f_{\#}(h))=\rho'([f\circ \alpha],[f\circ \beta])<\epsilon$.
\end{proof}
The following example illustrates that, in general, the topology on $\pinmet(X,x_0)$ induced by the pseudometric $\rho$ may vary with our original choice of metric on $X$.
\begin{example}\label{puncturedplane}
The cylinder $X=\mathbb{R}\times S^n$ and punctured real space $Y=\mathbb{R}^{n+1}\backslash\{\mathbf{0}\}$ (with the Euclidean metrics) are homeomorphic and may be identified as topological spaces. However, the resulting pseudometrics on the $n$-th homotopy group induced non-equivalent group topologies. In particular, if $x_0=(0,d_0)$, then the resulting pseudometric $\rho_1$ on $\pi_n(\mathbb{R}\times S^n,x_0)$ is discrete. For a proof one could apply Lemma \ref{discretenesslemma} below to see that $\pinmet(S^n,d_0)$ is discrete and then apply Proposition \ref{uniformlyinduced} to the projection map $\mathbb{R}\times S^n\to S^n$. On the other hand, let $\rho_2$ denote the resulting pseudometric on $\pi_n(Y,d_0)$. For $n\geq 2$, let $\gamma_n:\ui\to Y$ to be the linear path from $d_0$ to $(1/n,0,\dots,0)$ and let $\alpha_n:S^n\to Y$ be the embedding of the n-sphere of radius $1/n$ centered at the origin. Now consider the path-conjugates $\gamma_n\ast\alpha_n$, all of which represent a generator $g$ of $\pi_n(Y,y_0)\cong \mathbb{Z}$. The path-conjugate $\gamma_n\ast c_{n}$ is null-homotopic (where $c_n\in \Omega^n(Y,\gamma_n(1))$ is the constant map) and \[\lim_{n\to \infty}\mu(\gamma_n\ast c_n,\gamma_n\ast \alpha_n)=0.\]Thus $\rho_2(g,1)=0$. It follows that the resulting pseudometric group $\pinmet(Y,d_0)$ is indiscrete.
\end{example}
\begin{theorem}\label{independentthm}
If $X$ is a path-connected, compact, metrizable topological space and $x_0\in X$, then the homeomorphism type of $\pinmet(X,x_0)$ is independent of the choice of metric on $X$.
\end{theorem}
\begin{proof}
Suppose metrics $d_1$ and $d_2$ both induced the topology of $X$. Let $\rho_1$ and $\rho_2$ be the respective pseudometrics on $\pi_n(X,x_0)$. Since $X$ is compact, the identity maps $id:(X,d_1)\to (X,d_2)$ and $id:(X,d_2)\to (X,d_1)$ are uniformly continuous. By \Cref{uniformlyinduced}, the induced identity homomorphisms $(\pi_n(X,x_0),\rho_1)\to (\pi_n(X,x_0),\rho_2)$ and $(\pi_n(X,x_0),\rho_2)\to (\pi_n(X,x_0),\rho_1)$ (with the respective pseudometrics) are continuous and thus inverse isomorphisms.
\end{proof}
Next, we observe that, in general, the isomorphism class of the topological group $\pinmet(X,x_0)$ does not depend on the choice of basepoint. Fix a retraction $r:S^n\times \ui\to S^n\times \{0\}\cup \{d_0\}\times \ui$. For any path $\gamma:\ui\to X$ and map $\alpha:(S^n,d_0)\to (X,\gamma(1))$, we define the \textit{path-conjugate} $\gamma\ast \alpha:(S^n,d_0)\to (X,\gamma(0))$ to be the composition of $r(d,0):S^n\to S^n\times \{0\}\cup \{d_0\}\times \ui$ followed by applying $\gamma$ to $\{d_0\}\times \ui$ and $\alpha$ to $ S^n\times \{0\}$. This defines a natural map $\Omega^n(X,\gamma(1))\to \Omega^n(X,\gamma(0))$, $\beta\mapsto \gamma\ast\beta$, which induces a change-of-basepoint isomorphism $\Gamma:\pi_n(X,\gamma(1))\to \pi_n(X,\gamma(0))$ on homotopy classes. Moreover, when $\gamma$ is a loop, the action of loops on $n$-loops defines a jointly continuous map $\Omega(X,x_0)\times \Omega^n(X,x_0)\to \Omega^n(X,x_0)$, $(\gamma,\alpha)\mapsto\gamma\ast\alpha$. This map induces the usual $\pi_1$-action $[\gamma]\ast[\alpha]=[\gamma\ast\alpha]$ on $\pi_n(X,x_0)$.

\begin{proposition}\label{isomprop}
For any path $\gamma:\ui\to X$, the group isomorphism \[\Gamma:\pinmet(X,\gamma(1))\to \pinmet(X,\gamma(0)),\] $\Gamma([\beta])=[\gamma\ast \beta]$ is an isometry.
\end{proposition}
\begin{proof}
For all $\alpha,\beta\in \Omega^n(X,\gamma(1))$, we have $\mu(\alpha,\beta)=\mu (\gamma\ast \alpha,\gamma\ast \beta)$ and thus \[\rho([\gamma\ast \alpha],[\gamma\ast \beta])\leq \rho([\alpha],[\beta]).\] Thus $\Gamma$ is non-expansive. The inverse $\Gamma^{-1}:\piomet(X,\gamma(0))\to \piomet(X,\gamma(1))$, $\Gamma^{-1}([\beta])= [\gamma^{-}\ast\beta]$ is non-expansive for the same reason (replacing $\gamma$ with $\gamma^{-}$). Thus $\Gamma$ is an isometry.
\end{proof}
\begin{theorem}
The $\pi_1$-action $\piomet(X,x_0)\times \pinmet(X,x_0)\to \pinmet(X,x_0)$ given by $([\gamma],[\alpha])\mapsto [\gamma\ast\alpha]$ is jointly continuous. Moreover, $\piomet(X,x_0)$ acts on $\pinmet(X,x_0)$ by isometry.
\end{theorem}
\begin{proof}
That the action is by isometry follows immediately from Proposition \ref{isomprop}. For continuity, suppose $\{g_k\}_{k\in\bbn}\to g$ in $\piomet(X,x_0)$ and $\{a_k\}_{k\in\bbn}\to a$ in $\pinmet(X,x_0)$. Let $\epsilon>0$. Find $K\in\bbn$ such that $\max\{\rho(g_k,g),\rho(a_k,a)\}<\epsilon/2$ for all $k\geq K$. Thus for $k\geq K$, there exist $\gamma_k\in g_k$ and $\delta_k\in g$ with $\mu(\gamma_k,\delta_k)<\epsilon/2$ and $\alpha_k\in a_k$ and $\beta_k\in a$ with $\mu(\alpha_k,\beta_k)<\epsilon/2$. Our definition of the action $\ast $ of loops on $n$-loops using a fixed retraction $r$ ensures that $\mu(\gamma_k\ast \alpha_k,\delta_k\ast \beta_k)<\epsilon/2$. Since $g_k\ast a_k=[\gamma_k\ast \alpha_k]$ and $g\ast a=[\delta_k\ast\beta_k]$ for all $k\geq K$, we have $\rho(g_k\ast a_k,g\ast a)<\epsilon$. We conclude that $\{g_k\ast a_k\}_{k\in\bbn}\to g\ast a$ in $\pinmet(X,x_0)$.
\end{proof}
Finally, we compare the pseudometric topology to the quotient and $\tau$-topologies.
\begin{proposition}\label{finerthan}
The function $\pi:\loopnx\to \pinmet(X,x_0)$, $\pi(\alpha)=[\alpha]$ is continuous. Thus the topologies of $\pinqtopx$ and $\pi_{n}^{\tau}(X,x_0)$ are at least as fine as that of $\pinmet(X,x_0)$.
\end{proposition}
\begin{proof}
Suppose $\{\alpha_k\}_{k\in\bbn}\to \alpha$ in $\loopnx$ and $\epsilon>0$. There is an $K\in\mathbb{N}$ such that $\mu(\alpha,\alpha_k)<\epsilon/2$ for $k\geq K$. Thus $\rho([\alpha],[\alpha_k])<\epsilon$ for $k\geq K$ showing that $\{[\alpha_k]\}_{k\in\bbn}\to [\alpha]$ in $\pinmet(X,x_0)$. Thus $\pi$ is continuous. The second statement follows directly from the characterizations of the quotient and $\tau$-topologies in Section \ref{sectiontopologies}.
\end{proof}
\begin{remark}
Although basepoint-change isomorphisms are continuous when $\pi_n$ is given the quotient and $\tau$-topologies, the $\pi_1$-action $\pi_1(X,x)\times \pi_n(X,x)\to \pi_n(X,x)$ can be discontinuous when $\pi_1$ and $\pi_n$ are given the quotient topology \cite{Brazpioneaction}. Apparently, it is unknown if the $\pi_1$-action $\pi_1(X,x)\times \pi_n(X,x)\to \pi_n(X,x)$ is continuous in the $\tau$-topology.
\end{remark}

\section{Comparison with the Shape Topology}\label{sectioncomparewithshape}

\begin{lemma}\label{discretenesslemma}
If $K$ is a finite polyhedron and $n \geq 1$, then $\pinshape(K,x_0)$ is discrete.
\end{lemma}

\begin{proof}
If the compact metric space $(K,d)$ is the underlying space of a finite simplicial complex, then $K$ admits a topologically compatible $CAT(1)$ metric $d'$ \cite[Corollary 5.19]{BH}. Consequently there exists $\epsilon>0$ so that for any space $Y$, if the maps $f:Y\to K$ and $g:Y\to K$ are uniformly close with $d(f(y),g(y))<\epsilon$ for all $y \in Y$, then $f$ and $g$ are canonically homotopic via a homotopy $H$. The homotopy $H$ maps $f(y)$ to $g(y)$ linearly with time, along the unique geodesic $[f(y),g(y)]$ in $(K,d')$. The existence and continuity of $H$ follows from Proposition 1.4 and Corollary 3.13 of \cite{BH}. Thus $\pinmet(K,x_0)$ is discrete with the induced pseudometric generated by $d'$. Since $K$ is compact, Theorem \ref{independentthm} gives that the topology of $\pinmet(K,x_0)$ is independent of the choice of metric on $K$. Hence, $\pinmet(K,x_0)$ is discrete with the induced pseudometric generated by $d$.
\end{proof}

\begin{proposition}\label{metvshape}
If $X$ is a compact metric space, then the topology of $\pinmet(X,x_0)$ is at least as fine as that of $\pinshape(X,x_0)$.
\end{proposition}
\begin{proof}
Since $X$ is compact, we may construct $\check{\pi}_{n}(X,x_0)$ by replacing $\cov(X)$ with a cofinal sequence of finite covers. Whenever $\scru\in\cov(X)$ is finite, $|N(\scru)|$ is a finite polyhedron and thus $\pinmet(|N(\scru)|,U_0)$ is discrete by Lemma \ref{discretenesslemma}. Since $X$ is compact, any canonical map $p_{\scru}:X\to |N(\scru)|$ will be uniformly continuous and thus the induced homomorphism $p_{\scru\#}:\pinmet(X,x_0)\to \pinmet(|N(\scru)|,U_0)$ will be continuous by Proposition \ref{uniformlyinduced}. Since $\pinshape(X,x_0)$ has the initial topology with respect to homomorphisms $p_{\scru\#}$ with discrete codomains, it follows that the topology of $\pinmet(X,x_0)$ is at least as fine as that of $\pinshape(X,x_0)$.
\end{proof}
\begin{remark}\label{noncompact}
The previous proposition fails when $X$ is no longer required to be compact. If $X=\mathbb{R}^{n+1}\backslash\{\bf{0}\}$ is punctured real $(n+1)$-space as in \Cref{puncturedplane}, then $\pinmet(X,x_0)$ is isomorphic to $\bbz$ with the indiscrete topology whereas $\pinshapex $ is isomorphic to $\bbz$ with the discrete topology.
\end{remark}
\begin{definition}
Let $n\geq 1$ and $\scru$ be an open cover of $X$. The \textit{$n$-th Spanier group of $(X,x_0)$ with respect to $\scru$} is the subgroup $\span(\scru,x_0)$ of $\pi_n(X,x_0)$ generated by path-conjugates $[\gamma\ast f]\in\pi_n(X,x_0)$ where $\gamma:(\ui,0)\to (X,x_0)$ is a path and $\alpha:(S^n,d_0)\to (X,\gamma(1))$ has image in $U$ for some $U\in\scru$. The \textit{$n$-th Spanier group} of $(X,x_0)$ is the intersection $\span(X,x_0)=\bigcap_{\scru\in \cov(X)}\span(\scru,x_0)$.
\end{definition}

Spanier groups (in dimension $n=1$) first appeared in \cite{Spanier66} and have been used frequently in the past decade. Higher Spanier groups were more recently introduced \cite{BKPnspanier}. Note that since Spanier groups $\span(\scru,x_0)$, $\cov(X)$ are always normal subgroups of $\pi_n(X,x_0)$, the sets of left cosets $\{[g]\span(\scru,x_0)\mid \scru\in\cov(X),[g]\in\pi_n(X,x_0)\}$ forms a basis for a topology on $\pi_n(X,x_0)$.

\begin{definition}\label{defSpaniertopology}
We refer to the topology on $\pi_n(Y,y_0)$ generated by cosets of the Spanier groups $\span(\scru,y_0)$ as the \textit{Spanier topology}.
\end{definition}

In the current paper, the Spanier topology will not be studied in detail but will only serve as a convenient bound for the pseudometric topology. In order to make these comparisons, we recall known results that compare Spanier groups with homomorphisms $p_{\scru\#}:\pi_n(X,x_0)\to \pi_n(|N(\scru)|,U_0)$ induced by canonical maps. The following has a straightforward proof. See \cite[Prop 4.13]{BKPnspanier} or \cite[Lemma 3.10]{AcetiBrazas}.

\begin{proposition}
For every $\scru\in\cov(X)$ and choice of canonical map $p_{\scru}:X\to |N(\scru)|$, there exists a $\scrv\in \cov(X)$ such that $\span(\scrv,x_0)\leq \ker (p_{\scru\#}:\pi_n(X,x_0)\to \pi_n(|N(\scru)|,U_0))$.
\end{proposition}

Since the shape topology on $\pi_n(X,x_0)$ is generated by left cosets of subgroups of the form $\ker (p_{\scru\#})$, we have the following.

\begin{corollary}\cite[Remark 5.5]{AcetiBrazas}\label{spanierfinerthenshape}
For any metrizable space $X$, Spanier topology on $\pi_n(X,x_0)$ is at least as fine as the shape topology.
\end{corollary}

Determining when the Spanier topology agrees with the shape topology is a more technical matter that has been addressed in \cite{AcetiBrazas}. To state this properly, we recall the following definition.

\begin{definition}
Let $n\geq 0$. A topological space $Y$ is $LC^n$ at $y\in Y$ if for every neighborhood $U$ of $y$, there exists a neighborhood $V$ of $y$ such that $V\subseteq U$ and such that for all $0\leq k\leq n$, every map $f:S^k\to V$ extends to a map $g:D^{k+1}\to U$. We say $Y$ is $LC^n$ if $Y$ is $LC^n$ at all of its points.
\end{definition}

The following lemma was proved in dimension $n=1$ in \cite{BF}. For $n\geq 2$, the proof requires other techniques from shape theory.

\begin{lemma}\cite[Lemma 5.1]{AcetiBrazas}\label{technical}
Suppose $X$ is $LC^{n-1}$. Then for every $\scru\in\cov(X)$, there exists $\scrv\in \cov(X)$ such that for any canonical map $p_{\scrv}:X\to |N(\scrv)|$, we have $\ker (p_{\scrv\#}:\pi_n(X,x_0)\to \pi_n(N(\scrv),V_0))\leq \span(\scru,x_0)$.
\end{lemma}

Corollary \ref{spanierfinerthenshape} and Lemma \ref{technical} now combine to give the following.

\begin{corollary}\label{spanvshape}
If $X$ is a $LC^{n-1}$ metrizable space, then the Spanier and shape topologies on $\pi_n(X,x_0)$ agree.
\end{corollary}
Next, we compare the Spanier topology and pseudometric topologies by comparing neighborhood bases at the identity element.

\begin{proposition}
Let $(X,d)$ be a path-connected metric space. Then for every $r>0$, there exists $\scru\in\cov(X)$ such that $\span(\scru,x_0)\leq B_{\rho}(e,r)$.
\end{proposition}

\begin{proof}
Given $r>0$, let $\scru=\{B_d(x,r/2)\mid x\in X\}$ be the cover of $X$ by $r/2$-balls. Consider a generator $[\alpha\ast f]\in \span(\scru,x_0)$ where $\alpha:(\ui,0)\to (X,x_0)$ is a path and $f:(S^n,d_0)\to (X,\alpha(1))$ is a map with image in $B_d(x,r/2)$ for some $x\in X$. Let $g:S^n\to X$ be the constant map at $\alpha(1)$. Since $e=[\alpha\ast g]$ and $\mu(\alpha\ast f,\alpha\ast g)<r$, we have $\rho([\alpha\ast f],e)<r$, giving $[\alpha\ast f]\in B_{\rho}(e,r)$.
\end{proof}

\begin{corollary}\label{spanvmetric}
For a metric space $(X,d)$, the Spanier topology on $\pi_n(X,x_0)$ is at least as fine as the pseudometric topology induced by $d$.
\end{corollary}

At this point, we have sufficient results to prove Case (1) of Theorem \ref{mainthm}. The following lemma addresses Case (2).
\begin{lemma}\label{nestedretractlemma}
If $(X,d)$ is homeomorphic to an inverse limit $\varprojlim_{j\in\bbn}(X_j,r_{j+1,j})$ of finite polyhedra $X_j$ where the bonding maps $r_{j+1,j}:X_{j+1}\to X_j$ are retractions, then for all $n\geq 1$, the shape topology on $\pi_n(X,x_0)$ is at least as fine as the pseudometric topology.
\end{lemma}
\begin{proof}
Identify $X=\varprojlim_{j\in\bbn}(X_j,r_{j+1,j})$ and let $r_j:X\to X_j$, $j\in\bbn$ denote the projection maps. Let $s_{j,j+1}:X_j\to X_{j+1}$ be a section to $r_{j+1,j}$. When $i>j$, $s_{j,i}=s_{i-1,i}\circ s_{i-2,i-1}\circ\cdots\circ s_{j,j+1}$ is a section to $r_{i,j}=r_{j,j+1}\circ\cdots \circ r_{i-2,i-2}\circ r_{i,i-1}$. For fixed all pairs $i,j\in \bbn$, let $t_{j,i}:X_j\to X_i$ be $r_{j,i}$ if $j> i$, $id_{X_j}$ if $i=j$, and $s_{j,i}$ if $j<i$. Then the maps $\{t_{j,i}\}_{i\in\bbn}$ induce a map $t_j:X_j\to X$ which is a section to $r_j$. Thus we may identify $X_1\subseteq X_2\subseteq X_3\subseteq \cdots$ as a nested sequence of closed subspaces of $X$ where $\overline{\bigcup_{j\in\bbn}X_j}=X$. In particular $X_k$ consists of the sequences $(x_j)_{j\in\bbn}\in X$ where $x_j=x_k$ for all $j\geq k$ and $r_k((x_j)_{j\in\bbn})=(x_1,x_2,\dots,x_{k-1},x_k,x_k,x_k,\dots)$ for all $(x_j)_{j\in\bbn}\in X$.

Pick basepoint $x_0\in X_1$ and fix $n\geq 1$. Let $\delta>0$ so that $B_{\rho}(e,\delta)$ is a basic neighborhood of the identity in $\pinmet(X,x_0)$. Since $X$ is compact, the sequence $\{X_j\}_{j\in\bbn}$ gives an $\bf{HPol_{\ast}}$-expansion of $X$ and therefore, the subgroups $H_j=\ker(r_{j\#}:\pi_n(X,x_0)\to \pi_n(X_j,x_0))$ form a neighborhood base at the identity in $\pinshape(X,x_0)$. To show that the shape topology is finer than (or equal to) the pseudometric topology it suffices to show that there exists $k\in\bbn$ such that $H_k\subseteq B_{\rho}(e,\delta)$.

First, we claim that there exists $k\in\bbn$ such that $d(x,r_k(x))<\frac{\delta}{2}$ for all $x\in X$. By Theorem \ref{independentthm}, we may alter the metric $d$ on $X$ without changing the topology of $\pinmet(X,x_0)$. Since $\pinshape(X,x_0)$ is a topological invariant, doing so will not change the topology of $\pinshape(X,x_0)$ either. In particular, let $d_j$ be a metric inducing the topology of $X_j$ that is bounded by $1$. For elements $x=(x_1,x_2,\dots)$ and $y=(y_1,y_2,\dots)$ of $\prod_{j=1}^{\infty}X_j$, the formula $d(x,y)=\sum_{j=1}^{\infty}\frac{d_j(x_j,y_j)}{2^j}$ defines a metric that induces the topology of $\prod_{j=1}^{\infty}X_j$ with the product topology \cite[4.2.2]{Eng89}. Since $X$ is topologized a subspace of $\prod_{j=1}^{\infty}X_j$, $d(x,y)$ restricts to a metric that induces the topology of $X$. Find $k\in\bbn$ such that $\sum_{j=k+1}^{\infty}\frac{1}{2^j}<\frac{\delta}{2}$. Then, for any $x=(x_1,x_2,\dots)\in X$, we have \[d(x,r_k(x))=\sum_{j=k+1}^{\infty}\frac{d_j(x_j,x_k)}{2^j}\leq \sum_{j=k+1}^{\infty}\frac{1}{2^j}<\frac{\delta}{2}.\] This proves the desired claim.

Finally, we check that $H_k\subseteq B_{\rho}(e,\delta)$. Given $a\in H_k$, find an $n$-loop $\alpha\in a$. Then $\beta=r_k\circ \alpha$ is null-homotopic in $X_k$ and since $X_k\subseteq X$, $\beta$ is null-homotopic in $X$. Moreover, our choice of $k$ ensures that $d(\alpha({\bf{t}}),\beta({\bf{t})})=d(\alpha({\bf{t}}),r_k(\alpha({\bf{t}})))<\frac{\delta}{2}$ for all ${\bf{t}}\in\ui^n$. Thus $\mu(\alpha,\beta)<\frac{\delta}{2}$. It follows that $\rho(a,e)<\delta$, giving $a\in B_{\rho}(e,\delta)$. The inclusion $H_J\subseteq B_{\rho}(e,\delta)$ follows.
\end{proof}
At this point, we have established all results needed to prove Theorem \ref{mainthm}.
\begin{proof}[Proof of Theorem \ref{mainthm}]
Suppose $(X,d)$ is a compact metric space. By \Cref{metvshape}, the pseudometric topology is at least as fine as the shape topology (note that this is where compactness is required). To prove equivalence of the two topologies in Case (1), we assume $X$ is $LC^{n-1}$. By Corollary \ref{spanvmetric}, the Spanier topology is at least as fine as the pseudometric topology. Thus, we have the following situation: 
\[ \text{shape topology}\,\subseteq\,\text{pseudometric topology}\,\subseteq\,\text{Spanier topology}.\]
Finally, Corollary \ref{spanvshape} implies that the Spanier topology agrees with the shape topology. Thus the pseudometric topology agrees with the shape topology. Equivalence in Case (2) follows directly from \Cref{nestedretractlemma}.
\end{proof}
\begin{remark}
The $LC^{n-1}$ condition in Case (1) of Theorem \ref{mainthm} implies the well-known $n$-movability condition from shape theory \cite[II\S 8.1, Theorem 6]{MS82}. Moreover, the hypothesis in Case (2) of Theorem \ref{mainthm} directly implies that $X$ is $n$-movable. However, Theorem \ref{mainthm} does not appear to generalize readily to $n$-movable spaces. Indeed, the main structure of interest is the homotopy group $\pi_n(X,x_0)$ itself, rather than an $\bf{HPol_{\ast}}$-expansion or the shape-type of $X$. In general, movability conditions allow one to lift maps inductively along an inverse system. However, for the standard shape category these lifts are only lifts up to homotopy. One must be able to obtain strict lifts in order to induce a map to the limit.
\end{remark}
Next, we consider an example where the equivalence of the shape topology and pseudometric topology allows one to characterize the pseudometric topology by appealing to the simpler shape-theoretic setting.
\begin{example}[Earring mapping tori]
Fix $n\geq 2$. The \textit{$n$-dimensional infinite earring space} is $\bbe_n=\bigcup_{k\in\bbn}C_k$ where $C_k\subseteq \bbr^{n+1}$ is the $n$-sphere of radius $1/k$ centered at $(1/k,0,0,\dots,0)$ \cite{EK00higher}. Sometimes this space is referred to as the \textit{Barratt-Milnor Spheres} \cite{BarrattMilnor}. It is known that $\bbe_n$ is $(n-1)$-connected, locally $(n-1)$-connected (and thus $LC^{n-1}$), and $\pi_n(\bbe_n,\bf{0})\cong \bbz^{\bbn}$ \cite{EK00higher}. Let $f:\bbe_n\to\bbe_n$ be the canonical based ``shift map" that maps $C_k$ homeomorphically to $C_{k+1}$ for all $k\geq 1$ and $T_f=\bbe_n\times\ui/\mathord{\sim}$, $(x,0)\sim (f(x),1)$ be the mapping torus of $f$ (See Figure \ref{fig1}). Then $T_f$ is an $n+1$-dimensional Peano continuum. Take the image of the origin to be the basepoint $x_0$. The group $\pi_n(T_f,x_0)$ is uncountable but its elements have a fairly simple characterization \cite[Example 5.3]{AcetiBrazas}. We use Theorem \ref{mainthm} to characterize the pseudometric topology on $\pi_n(T_f,x_0)$. Since $T_f$ is shape equivalent to the aspherical space $S^1$, the $n$-th shape homotopy group $\check{\pi}_n(T_f,x_0)$ is trivial. Thus $\pinshape(T_f,x_0)$ is an indiscrete group. Given a point $x\in T_f$, either $T_f$ is locally contractible at $x$ or $x$ admits a neighborhood that deformation retracts onto a subspace that is homeomorphic to $\bbe_n$. Since $\bbe^n$ is $LC^{n-1}$, it follows that $T_f$ is $LC^{n-1}$. It then follows from Case (1) of Theorem \ref{mainthm} that the topology of $\pinmet(T_f,x_0)$ agrees with the shape topology and is therefore indiscrete.
\end{example}

\begin{figure}[h]
\centering \includegraphics[height=2.3in]{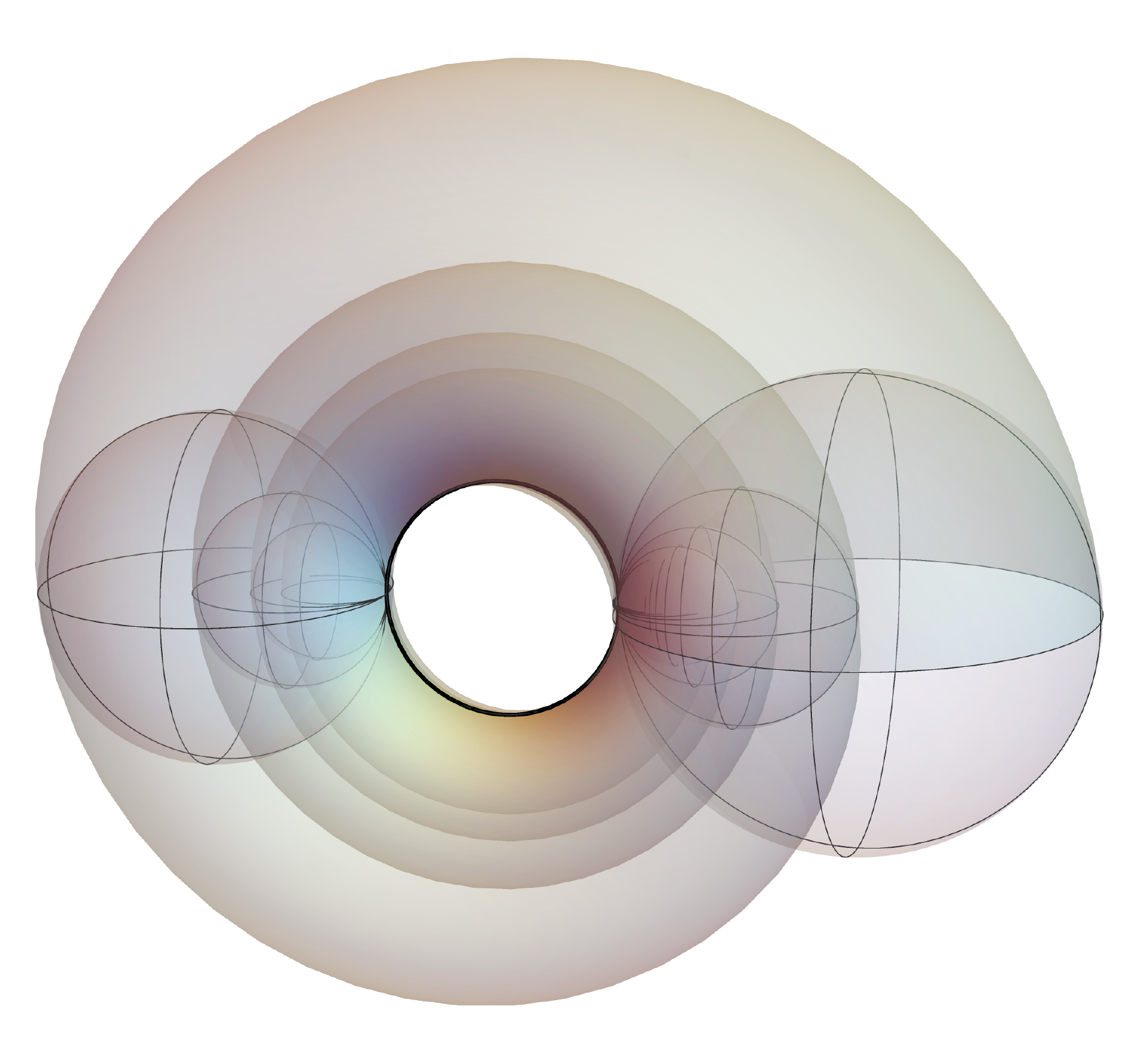}
\caption{\label{fig1}The mapping torus of the shift map $f:\bbe_2\to\bbe_2$ of the $2$-dimensional earring space.}
\end{figure}
We also consider an example showing that Theorem \ref{mainthm} fails to hold in any dimension $n\geq 1$ without the local path-connectedness hypothesis, i.e the assumption that $X$ is $LC^{0}$.
\begin{example}\label{cylinderexample}
Fix $n\geq 1$. Let $A_0=\{0\}\times [-1,1]\subseteq \bbr^2$ and $A_1=\{(x,\sin(1/x))\in\bbr^2\mid 0<x\leq 1/\pi\}$ so that $A=A_0\cup A_1$ is the closed topologists sine curve. Let $a_0=(0,1)\in A_0$ and $a_1=(1/\pi,0)\in A_1$. Consider the quotient space $B=A\times S^n/\{a_1\}\times S^n$ with quotient map $q:A\times S^n\to B$ (See Figure \ref{fig2}) and let $C=q(A_0\times S^n)$ be the cylinder over $S^n$. Set $x_0=q(a_0,d_0)$ and $x_1=q(\{a_1\}\times S^n)$. For each $a\in A$, let $\ell_a:S^n\to B$ be the map $\ell_a(t)=q(a,t)$, which is an embedding when $a\neq a_1$ and constant at $x_1$ if $a=a_1$. Note that $\ell_a$ is null-homotopic if and only if $a\in A_1$. Construct the space $X$ by attaching an arc $\ui$ to $B$ by identifying $0\sim x_0$ and $1\sim x_1$. We metrize $X$ with the canonical quotient metric. The space $X$ is a compact metric space, which may be embedded in $\bbr^{n+2}$. However, $X$ is not $LC^0$ at any point in $C$. Note that $\pi_n(X,x_0)$ is infinite cyclic generated by $[\ell_{a_0}]$. However, $X$ is shape equivalent to the aspherical space $S^1$ and the $n$-th shape homomorphism $\Psi:\pi_n(X,x_0)\to \check{\pi}_n(X,x_0)$ is the trivial homomorphism for all $n\geq 1$ (the domain and codomain are both $\bbz$ in the case $n=1$). Therefore, $\pinshape(X,x_0)$ is indiscrete. 

On the other hand, we check that $\pinmet(X,x_0)$ is discrete. Let $Y$ be the union of the cylinder $C$ and the added arc. then $Y$ is a compact polyhedron homotopy equivalent to $S^n$ and so $\pinmet(Y,x_0)$ is discrete. Therefore, there exists an $\epsilon>0$ such that if a based map $f:S^n\to Y$ is $\epsilon$-close to some null-homotopic based map (in the uniform metric), then $f$ is null-homotopic. Choose such an $\epsilon$ which is also less than one third the diameter of the attached arc. Thus an $\epsilon$-neighborhood of $x_1$ (the endpoint of the arc) is disjoint from $C$. Now suppose $f:S^n\to X$ is a null-homotopic based map and that $g:S^n\to X$ is $\epsilon$-close to $f$. Define $f':S^n\to X$ to agree with $f$ on $f^{-1}(Y)$ and to map $f^{-1}(X\backslash Y)$ to $x_1$. We define $g'$ similarly for $g$. It's easy to see that $f\simeq f'$, $g\simeq g'$, and that $f',g'$ are $\epsilon$-close maps in $Y$. Therefore, $f'\simeq g'$ in $Y$. We conclude that $g$ is null-homotopic in $X$. This proves that $\pinmet(X,x_0)$ is discrete.
\end{example}

\begin{figure}[h]
\centering \includegraphics[height=2.3in]{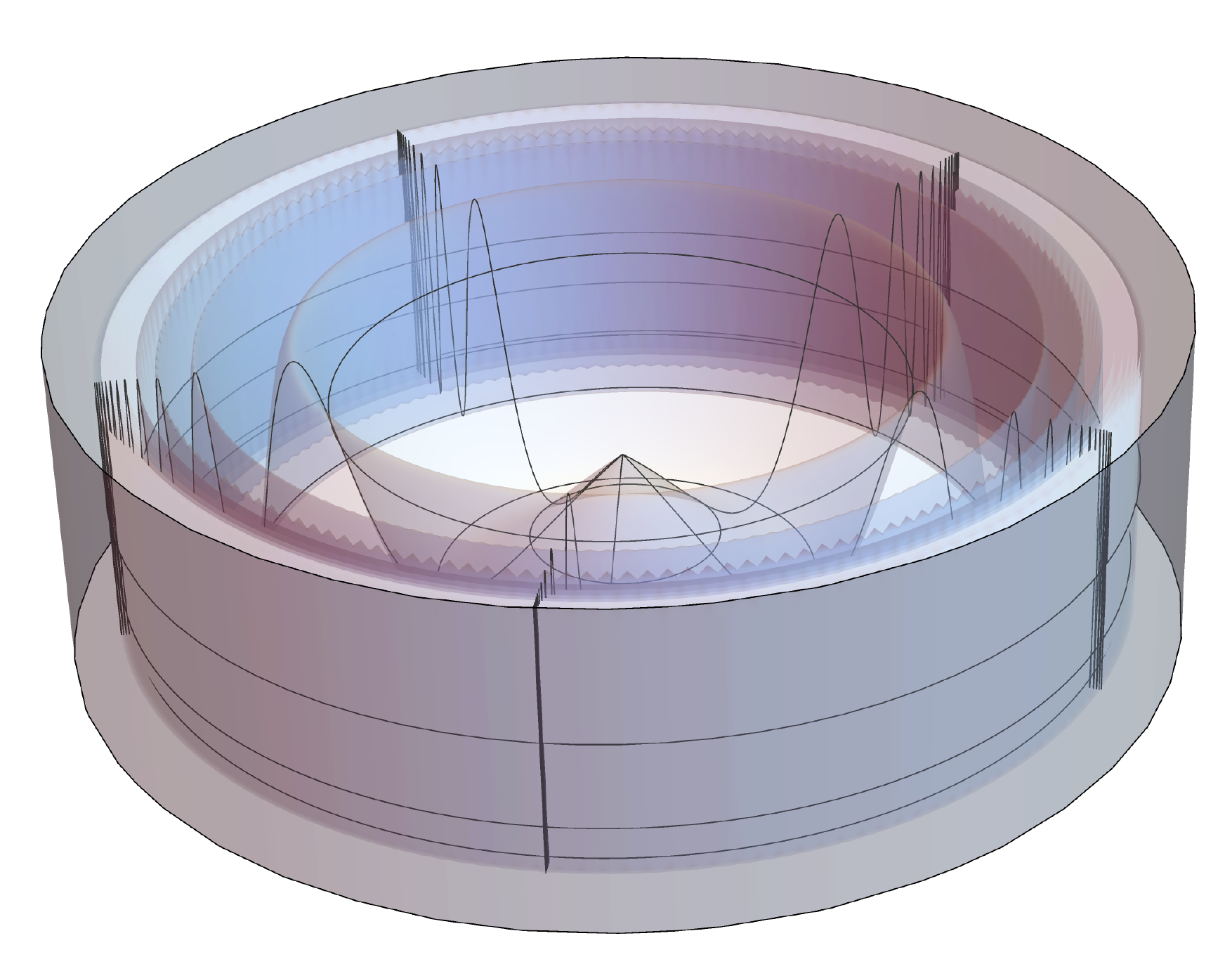}
\caption{\label{fig2}The space $B$ in the case $n=1$ is the union of a cylinder over $S^1$ and a non-compact surface that limits on the cylinder. The space $X$ is constructed by connecting the two path components with an arc.}
\end{figure}
In the next example, we illustrate an application of Case (2) of Theorem \ref{mainthm} as an important situation where Case (1) does not apply.
\begin{example}\label{shrinkingwedge}
Let $K_1,K_2,K_3,\cdots$ be a sequence of based, compact polyhedra and $X_k=\bigvee_{j=1}^{k}K_j$ for $k\in\bbn$. We take $r_{k+1,k}:X_{k+1}\to X_k$ to be the retraction which collapses $K_{k+1}$ to the basepoint and $X=\varprojlim_{k\in\bbn}(X_k,r_{k+1,k})$. Then $X=\widetilde{\bigvee}_{j\in\bbn}K_j$ is a Peano continuum, which we refer to as the \textit{shrinking wedge} of the sequence $\{K_j\}_{j\in\bbn}$. Since $X$ is an inverse limit of a nested sequence of polyhedral retracts, Case (2) of Theorem \ref{mainthm} implies that for any choice of metric $d$ on $X$, $\pinmet(X,x_0)=\pinshape(X,x_0)$ for all $n\geq 1$. For example, the $m$-dimensional earring space $\bbe_m$ is homeomorphic to $\widetilde{\bigvee}_{j\in\bbn}S^m$ and the conclusion $\pinmet(\bbe_m,x_0)=\pinshape(\bbe_m,x_0)$ only follows from Case (1) of Theorem \ref{mainthm} when $n\leq m$. Interestingly, this equivalence of topologies holds even when infinitely many $K_j$ fail to be $m$-connected, in which case $X$ is not $LC^{m}$ and the algebraic structure of $\pi_n(X,x_0)$ is complicated significantly by Whitehead products.
\end{example}
We conclude by considering the special case of the fundamental group where covering spaces and their generalizations are typically useful.

\begin{corollary}\label{maincor}
For a Peano continuum $X$ and $x_0\in X$, the following are equivalent:
\begin{enumerate}
\item $\rho$ is a metric on $\pi_1(X,x)$,
\item $X$ is $\pi_1$-shape injective,
\item for every $[\alpha]\in \pi_1(X,x)$, there exists a covering map $p:(E,e)\to (X,x_0)$ such that $[\alpha]\notin p_{\#}(\pi_1(E,e))$, i.e $\alpha$ lifts to a non-loop in $E$.
\end{enumerate}
\end{corollary}
\begin{proof}
The equivalence (2) $\Leftrightarrow$ (3) was proved in \cite{BF}. (1) $\Rightarrow$ (2) if $\rho$ is a metric, then $\pi_{1}^{\met}(X,x_0)$ is Hausdorff. Since $\pi_{1}^{\met}(X,x_0)= \pishapex$ by Theorem \ref{mainthm}, $\pishapex$ is Hausdorff and it follows that $\Psi_1:\pi_{1}(X,x_0)\to\check{\pi}_{1}(X,x_0)$ is injective. (2) $\Rightarrow$ (1) If $\Psi_1:\pi_{1}(X,x_0)\to\check{\pi}_{1}(X,x_0)$ is injective, then $\pi_{1}^{\met}(X,x_0)=\pi_{1}^{\text{sh}}(X,x_0)$ is Hausdorff. Since the topology of $\pi_{1}^{\met}(X,x_0)$ is generated by the pseudometric $\rho$, $\rho$ is a metric. 
\end{proof}

\end{document}